\documentclass[preprint,12pt]{elsarticle}

\usepackage{amssymb,amsmath,latexsym, amsthm, amscd, mathrsfs, stmaryrd,  color}
\usepackage{epic}
\usepackage{CJK}
\usepackage{setspace}
\usepackage[all]{xy}
\usepackage{yhmath}
\usepackage{soul}
\usepackage{hyperref}
\usepackage{leftidx} 
\usepackage{mathabx} 
\usepackage{mathtools} 
\usepackage{rotating} 
\usepackage{stmaryrd} 
\usepackage{youngtab,young}

\usepackage{todonotes}
\usepackage{arydshln}
\makeatletter
\renewcommand*\env@matrix[1][*\c@MaxMatrixCols c]{%
  \hskip -\arraycolsep
  \let\@ifnextchar\new@ifnextchar
  \array{#1}}
\makeatother
\usepackage{pgf,tikz}

\usetikzlibrary{arrows, positioning, calc, chains}
\tikzset{
	ch/.style={circle,draw,on chain,inner sep=2pt},
	chj/.style={ch,join},
	every path/.style={shorten >=4pt,shorten <=4pt}
	}

\numberwithin{equation}{subsection}

\newtheorem{thm}{Theorem}[section]

\newtheorem{lem}[thm]{Lemma}
\newtheorem{rem}[thm]{Remark}

\newtheorem{prou}[thm]{Definition}
\newtheorem{example}[thm]{Example}

\journal{~~~~~~~~~~}

\begin{document}
\begin{CJK*}{GBK}{song}
\begin{spacing}{1.08}
\begin{frontmatter}



\title{The geometry  connectivity  of  hypergraphs}
\author[label1]{Chunli Deng}
\author[label2]{Lizhu Sun}
\author[label1,label2]{Changjiang Bu  \corref{cor}}\ead{buchangjiang@hrbeu.edu.cn}

\cortext[cor]{Corresponding author: {Changjiang Bu}}
\address{
\address[label1]{College of Automation, Harbin Engineering University, Harbin 150001, PR China}
\address[label2]{College of Mathematical Sciences, Harbin Engineering University, Harbin 150001, PR China}
}

\date{\today}



\begin{abstract}
Let $\mathcal{G}$ be a $k$-uniform hypergraph, $\mathcal{L}_{\mathcal{G}}$ be its Laplacian tensor.
And $\beta( \mathcal{G})$ denotes the maximum number of linearly independent nonnegative eigenvectors  of   $\mathcal{L}_{\mathcal{G}}$ corresponding to the   eigenvalue $0$.
In this paper, $\beta( \mathcal{G})$ is called  the geometry  connectivity of  $\mathcal{G}$.
We show that the   number of connected components of   $\mathcal{G}$ equals  the geometry  connectivity $\beta( \mathcal{G})$.

\end{abstract}


\end{frontmatter}

\section{Introduction}

In 1973,  Fiedler pointed that a graph $G$ is connected if and only if the second smallest eigenvalue of Laplacian matrix  $L_G$  is more than zero, this eigenvalue is called \textit{ algebraic connectivity} of graph $G$, denoted by $\alpha(G)$ \cite{fiedler1973}. Usually, this conclusion is called the Fiedler Theorem.
\begin{thm}\cite{fiedler1973}\label{Fiedler1}
A graph   $G$ is connected if and only if $\alpha(G)>0$.
\end{thm}
The Fiedler Theorem give the tight relation to the fundamental graph property and eigenvalues of graphs,  it attracts much attention and huge literatures followed.
In 1975, Fiedler further studied the  algebraic connectivity of graph $G$ in \cite{Fiedler1975}.
He showed that the eigenvector corresponding to $\alpha(G)$ induces partitions of the vertices of $G$ that are natural
connected clusters \cite{V.Nikiforov,jinfangrong}.
This property is important and efficient for partitioning of graphs.
After the publication of \cite{Fiedler1975},
the eigenvector corresponding to $\alpha(G)$ 
has been adopted by computer scientists and used in   algorithmic partitioning applications, see \cite{17,19}.

Further, the Fiedler  Theorem  can be generalized as follows.

\begin{thm}\cite{introduction,Spectra-of-graphs}\label{Fiedler2}
The   number of connected components of a
graph $G$ equals the algebraic multiplicity of  Laplacian eigenvalue 0 of $G $.
\end{thm}

The connectivity and the   number of connected components  hypergraphs are important topics as well.
But for the $k$-uniform hypergraph $\mathcal{G}$ $(k\geq3)$, from the Example 2.6,  we can see that the Theorem \ref{Fiedler1} and Theorem \ref{Fiedler2} can't be generalized to hypergraphs directly.
In \cite{Qilaplacian,liwei1,buperron}, the authors characterized the connectivity of hypergraphs by subtensors of the Laplacian tensor.

\vspace{3mm}
In this paper, we study the connectivity and the   number of connected components of $k$-uniform hypergraphs in terms of eigenvectors  of Laplacian tensor.
We give the concept of  the (Z-)geometry  connectivity  of $k$-uniform hypergraph $\mathcal{G}$ as follows.
\begin{prou}
The (Z-)geometry  connectivity  of $k$-uniform hypergraph $\mathcal{G}$, denoted by $\beta(\mathcal{G})$ ($\beta_Z(\mathcal{G})$),  is  defined as the maximum number of linearly independent nonnegative (Z-)eigenvectors  of Laplacian tensor  corresponding to the   (Z-)eigenvalue $0$.
\end{prou}
 We show that the   number of connected components of a $k$-uniform hypergraph $\mathcal{G}$ is  the (Z-)geometry  connectivity $\beta( \mathcal{G})$ ($\beta_Z(\mathcal{G})$).

\section{Preliminaries}

In this section, we introduce some concepts and lemmas.


For a positive integer $n$, denote $[n]=\{1,2,\ldots,n\}$.
Let  $\mathbb{R}^{[k,n]}$ and $\mathbb{R}_+^{[k,n]}$ denote the set of $k$-order $n$-dimensional  real tensors and nonnegative tensors, respectively. When $k=2$, the $\mathbb{R}^{[2,n]}$ (resp. $\mathbb{R}_+^{[2,n]}$) is  the set of all $n\times n$ real (resp. nonnegative) matrices.
Let $\mathbb{C}^{n}$, $\mathbb{R}^{n}$ and $\mathbb{R}_+^{n}$ denote the set of  $n$-dimensional complex vectors, real vectors and nonnegative vectors, respectively.

A tensor $\mathcal{A}=(a_{i_1{i_2} \cdots {i_k}})\in \mathbb{R}^{[k,n]}$ is called \textit{symmetric} if its each entry $a_{i_1{i_2} \cdots {i_k}}$ is invariant under any permutation of $i_1,{i_2}, \ldots ,{i_k}$.

A tensor $\mathcal{I}=(\delta_{i_1\cdots i_k})\in \mathbb{R}^{[k,n]}$ is called the \textit{identity tensor} if whose entry $\delta_{i\cdots i}=1$ for all $i$ and zero otherwise.

For  $\mathcal{A}=(a_{i_{1}i_{2}\cdots i_{m}}) \in \mathbb{R}_{+}^{[m,n]}$,
it is associated to a directed graph $G( \mathcal{A}) = (V(\mathcal{A}),E(\mathcal{A}))$, where $V(\mathcal{A}) = \{1,2, \ldots, n\}$ and  $ E(\mathcal{A})=\{ (i,j):a_{ii_2\cdots i_m }> 0,j\in \{i_2, \ldots, i_m\}\}$.
A nonnegative tensor $\mathcal{A} \in \mathbb{R}_{+}^{[m,n]}$
 is called \textit{weakly irreducible} if the associated directed graph
$G(\mathcal{A})$ is strongly connected (see  \cite{Friedland2013Perron,Pearson2014On}).

In 2005, the eigenvalue of tensors was proposed by Qi \cite{Qi05} and Lim \cite{lim}, independently.
Let $\mathcal{A}=(a_{i_{1}i_{2}\cdots i_{m}}) \in \mathbb{R}^{[m,n]}$.
If there exist $\lambda\in \mathbb{C}$ and a nonzero vector $x=(x_{1},\ldots,x_{n})^{\mathrm{T}} \in \mathbb{C}^{n}$ such that
\begin{align}\label{eq1'}
\mathcal{A}x^{m - 1}  = \lambda x^{\left[ {m - 1} \right]},\tag{2.1}
\end{align}
then $\lambda$ is called an \textit{eigenvalue} of $\mathcal{A}$ and $x$ is called an \textit{eigenvector} of $\mathcal{A}$ corresponding to $\lambda$, where
$\mathcal{A}x^{m - 1}\in \mathbb{C}^n$,
$(\mathcal{A} x^{m-1})_{i}=\sum\limits_{{i_2}, \ldots ,{i_m} =1}^{n} {{a_{i{i_2} \cdots {i_m}}}{x_{{i_2}}} \cdots {x_{{i_m}}}}$, $i\in[n]$, and $x^{[m-1]}=(x_{1}^{m-1}, \ldots ,x_{n}^{m-1})^{\mathrm{T}}$.
The $\rho(\mathcal{A})=\max\{|\lambda|: \lambda~ \text{is an eigenvalue of}~ \mathcal{A}\}$ is called the \textit{spectral radius} of  $\mathcal{A}$.
If there exist $\lambda\in \mathbb{R}$ and  a nonzero vector $x \in \mathbb{R}^{n}$ such that
\begin{align}\label{eq2'}
\mathcal{A}x^{m - 1}  = \lambda  x~~\text{and}~~x^{\mathrm{T}}x=1,\tag{2.2}
\end{align}
then $\lambda$ is called a \textit{Z-eigenvalue} of $\mathcal{A}$ and $x$ is called a \textit{Z-eigenvector} of $\mathcal{A}$ corresponding to $\lambda$.
Since the work of Qi \cite{Qi05} and Lim \cite{lim},
the research on eigenvalues of tensors and its applications has attracted much attention (see \cite{shu1,shu2,dengsto}).

In \cite{Friedland2013Perron,chperron,yy,yy2}, Perron-Frobenius theory of tensors were established. Next, we introduce some results on Perron-Frobenius theory of tensors that we used in this paper.

\begin{lem}\cite{yy}\label{non}
Let $\mathcal{A}\in \mathbb{R}_+^{[m,n]}$. If some eigenvalue of $\mathcal{A}$ has a positive eigenvector corresponding
to it, then this eigenvalue must be $\rho(\mathcal{A})$.
\end{lem}

\begin{lem}\cite{Friedland2013Perron,yy2} \label{perron}
Let  $\mathcal{A} \in \mathbb{R}_{+}^{[m,n]}$ be a  weakly irreducible tensor. Then $\rho(\mathcal{A})$ is an eigenvalue of $\mathcal{A}$ and there exists a unique positive eigenvector corresponding to $\rho(\mathcal{A})$ up to a multiplicative constant.
\end{lem}

\begin{lem}\cite{yy2}\label{non-irr}
Let $\mathcal{A}\in \mathbb{R}_+^{[m,n]}$  be a  weakly irreducible tensor. Suppose $x$ is an eigenvector corresponding to $\rho(\mathcal{A})$. Then $x$ contains no zero elements.
\end{lem}

\vspace{3mm}

Let a hypergraph ${\mathcal{G}}=(V({\mathcal{G}}),E({\mathcal{G}}))$, where ${V({\mathcal{G}})} = \{1,2, \ldots ,n\}$ and $E({\mathcal{G}}) = \{e_1,e_2, \ldots ,e_m\}$ are the vertex set and edge set of   ${\mathcal{G}}$, respectively. If each edge of ${\mathcal{G}}$ contains $k$  vertices, then ${\mathcal{G}}$ is called a \textit{$k$-uniform hypergraph}. Clearly, $2$-uniform hypergraphs are exactly the ordinary graphs.
The degree of a vertex $i$ of ${\mathcal{G}}$ is denoted by $d_i$, where $d_i=|\{e_j:i\in e_j,j=1,\ldots,m\}|$, $i\in[n]$. If all vertices of ${\mathcal{G}}$ have the same degree, then ${\mathcal{G}}$ is called \textit{regular}.
The \textit{adjacency tensor} \cite{cooper} of $k$-uniform hypergraph  ${\mathcal{G}}$, denoted by $\mathcal{A}_{\mathcal{G}}$, is a $k$-order $n$-dimensional nonnegative symmetric tensor with entries
\[
a_{i_1 i_2  \cdots i_k }  =\left\{
          \begin{array}{ll}
            \frac{1}{(k-1)!}, &{\text{if} ~\{i_1 ,i_2 , \ldots ,i_k \}\in E({\mathcal{G}});} \\
            0,  &{\text{otherwise}.}
          \end{array}
        \right.
\]
Let $\mathcal{L}_{\mathcal{G}}=\mathcal{D}_{\mathcal{G}}-\mathcal{A}_{\mathcal{G}}$ be the \textit{Laplacian tensor}  of  $\mathcal{G} $ \cite{Qilaplacian}, where $\mathcal{D}_{\mathcal{G}}$ is a diagonal tensor, whose diagonal entries are $d_1,\ldots,d_n$, respectively.



A \textit{path} $P$ in a $k$-uniform hypergraph ${\mathcal{G}}$  is defined to be an alternating sequence of vertices and edges $v_{0}e_{1}v_{1}e_{2} \cdots v_{l-1}e_{l}v_{l}$, where $v_0,\ldots,v_{l}$ (resp. $e_1,\ldots,e_l$) are distinct vertices (resp. edges) of ${\mathcal{G}}$ and $v_{i - 1} ,v_i  \in e_i $ for $i = 1, \ldots ,l$.
If there exists a path starting at $u$ and terminating at $v$ for all $u,v \in V({\mathcal{G}})$, then ${\mathcal{G}}$ is called \textit{connected}.
Let $X$ be a subset of $V({\mathcal{G}})$, ${\mathcal{G}}(X)$  denote the sub-hypergraph  of ${\mathcal{G}}$ induced by $X$.
If ${\mathcal{G}}(X)$ is connected and there isn't the paths  starting at the  vertices in $X$ and terminating at vertices in $V\setminus X$, then ${\mathcal{G}}(X)$ is called a  \textit{connected component} of ${\mathcal{G}}$.

\begin{lem} \cite{Pearson2014On}\label{connected}
Let ${\mathcal{G}}$ be a $k$-uniform hypergraph. Then ${\mathcal{G}}$ is connected if and only if adjacency tensor $\mathcal{A}_{\mathcal{G}}$ is weakly irreducible.
\end{lem}


In the following, we give an example to show that the Theorem \ref{Fiedler1} and Theorem \ref{Fiedler2} can't be generalized to hypergraphs directly.
\begin{example}
The Figure 1 is a $4$-uniform hypergraph $\mathcal{G}$,  by Theorem 4.3 in \cite{cooper}, we get that  the eigenvalues of $\mathcal{A}_{\mathcal{G}}$ are $0,-1,1,-i,i$ and the corresponding algebraic multiplicity are 36,16,16,16,16, respectively, where $i^2=-1$.

Since $\mathcal{L}_{\mathcal{G}}= \mathcal{D}- \mathcal{A}=\mathcal{I}- \mathcal{A}$, by the definition of eigenvalue of tensors, we have the eigenvalues of $\mathcal{L}_{\mathcal{G}}$ are $1,2,0,1+i,1-i$ and the corresponding algebraic multiplicity are 36,16,16,16,16, respectively.

Obviously, the number of connected component of ${\mathcal{G}}$ is 1. But the  algebraic multiplicity of Laplacian eigenvalue 0 is 16.
Thus,  the Theorem \ref{Fiedler1} and Theorem \ref{Fiedler2} can't be generalized to hypergraphs directly.
%

\vspace{3mm}

\begin{center}
\begin{tikzpicture}
\draw (11.5,0) ellipse (2.5cm and 0.8cm);
\draw (10,0) ellipse (0.2cm and 0.2cm);
\fill[black] (10,0) ellipse (0.2cm and 0.2cm);
\draw (11,0) ellipse (0.2cm and 0.2cm);
\fill[black] (11,0) ellipse (0.2cm and 0.2cm);
\draw (12,0) ellipse (0.2cm and 0.2cm);
\fill[black] (12,0) ellipse (0.2cm and 0.2cm);
\draw (13,0) ellipse (0.2cm and 0.2cm);
\fill[black] (13,0) ellipse (0.2cm and 0.2cm);
\end{tikzpicture}

\vspace{3mm}
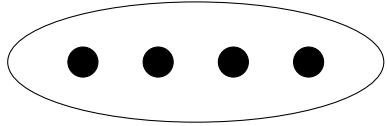
 Figure 1: 4-uniform hypergraph $\mathcal{G}$
\end{center}
\end{example}

\vspace{3mm}


\section{Main results}

In this section, we show that the   number of connected components of a $k$-uniform hypergraph $\mathcal{G}$ equals  the (Z-)geometry  connectivity $\beta( \mathcal{G})$ ($\beta_Z(\mathcal{G})$).
Before we show the main results, we first give the following result.
\begin{lem}\label{th1}
Let ${\mathcal{G}}$ be a connected $k$-uniform hypergraph.  Then $\beta(\mathcal{G})=1$.
\end{lem}
\begin{proof}
Let $\tilde{\mathcal{L}_{\mathcal{G}}}=\nabla \mathcal{I}-\mathcal{L}_{\mathcal{G}}$, where $\nabla$ is the maximum degree of $\mathcal{G}$.
It's easy to check that $(0,\textbf{e})$ is an eigenpair of $\mathcal{L}_{\mathcal{G}}$, where $\textbf{e}=(1,1,\ldots,1)^{\mathrm{T}}\in \mathbb{R}^n$. Then $(\nabla,\textbf{e})$ is an eigenpair of $\mathcal{\tilde{L}}_{\mathcal{G}}$.

Since ${\mathcal{G}}$ is connected, from Lemma \ref{connected}, $\mathcal{A}_{\mathcal{G}}$ is weakly irreducible.
So we have
 $\tilde{\mathcal{L}_{\mathcal{G}}}=\nabla \mathcal{I}-\mathcal{L}_{\mathcal{G}}=\nabla \mathcal{I}-\mathcal{D}_{\mathcal{G}}+\mathcal{A}_{\mathcal{G}}$ is a nonnegative weakly irreducible tensor.
Then, by Lemma \ref{non}, we get that $\nabla$ is the spectral radius of $\tilde{\mathcal{L}}_{\mathcal{G}}$.

From Lemma \ref{perron} and Lemma \ref{non-irr}, we get that the nonnegative eigenvector of a nonnegative  weakly irreducible tensor corresponding to  spectral radius is unique  up to a multiplicative constant.
So, $\textbf{e}$ is  a unique nonnegative eigenvector corresponding to  spectral radius $\nabla$ of $\tilde{\mathcal{L}}_{\mathcal{G}}$ up to a multiplicative constant.

Let $(0,x)$ is an eigenpair of $\mathcal{L}_{\mathcal{G}}$.
Obviously,
$$\tilde{\mathcal{L}_{\mathcal{G}}}x^{k-1}=(\nabla \mathcal{I}-\mathcal{L}_{\mathcal{G}})x^{k-1}=\nabla x^{[k-1]}-\mathcal{L}_{\mathcal{G}}x^{k-1}=\nabla x^{[k-1]}.$$
So  $(\nabla,x)$ is an eigenpair of $\mathcal{\tilde{L}}_{\mathcal{G}}$.
Hence,  $\textbf{e}$ is  a unique nonnegative eigenvector corresponding to  eigenvalue $0$ of $\mathcal{L}_{\mathcal{G}}$ up to a multiplicative constant, i.e., $\beta(\mathcal{G})=1$.
\end{proof}

Let $ \mathcal{A}=(a_{i_1\cdots i_k})\in \mathbb{R}^{[k,n]}$, $S$ be a subset of $[n]$, and $\mathcal{A}[S]=(a_{i_1\cdots i_k})$ denote a $k$-order $|S|$-dimensional \textit{subtensor} of $\mathcal{A}$, where $i_1,i_2,\ldots, i_k \in S$.

\begin{thm}\label{th2}
Let ${\mathcal{G}}$ be a $k$-uniform hypergraph. Then the number of connected components of $\mathcal{G}$ is the geometry connectivity $\beta(\mathcal{G})$.
\end{thm}
\begin{proof}
Denote $\mathcal{L}_{\mathcal{G}}$ the Laplacian  tensors of  ${\mathcal{G}}$ and $\textbf{e}=(1,1,\ldots,1)^{\mathrm{T}}\in \mathbb{R}^n$.
Let ${\mathcal{G}}(V_1),\ldots,{\mathcal{G}}(V_r)$ be the connected components of ${\mathcal{G}}$. Then subtensor $\mathcal{L}_{\mathcal{G}}[V_i]$ of $\mathcal{L}_{\mathcal{G}}$ is the Laplacian  tensors of sub-hypergraphs ${\mathcal{G}}(V_i)$, $i=1,\ldots,r$.
Then we have
\begin{align}\label{eq3.1}
\mathcal{L}_{\mathcal{G}}[V_i](\textbf{e}[V_i])^{k-1}=0, ~i=1,\ldots,r.\tag{3.1}
\end{align}
Thus, $\textbf{e}[V_i]$ is  a  nonnegative eigenvector corresponding to  eigenvalue $0$ of $\mathcal{L}_{\mathcal{G}}[V_i]$, $i=1,\ldots,r$.

Let  $\textbf{e}_{V_i}=(x^{(i)}_j)\in \mathbb{R}^n$, its entry $x^{(i)}_j=1$ if $j\in V_i$ and zero otherwise, $i=1,\ldots,r$.
Then by (\ref{eq3.1}), we have
\[
\mathcal{L}_{\mathcal{G}}(\textbf{e}_{V_i})^{k-1}=0,~i=1,\ldots,r,
\]
i.e., $(0,\textbf{e}_{V_1}), \ldots, (0,\textbf{e}_{V_r})$ are    eigenpairs of $\mathcal{L}_{\mathcal{G}}$.
And since $V_i\bigcap V_j=\varnothing$, $i,j=1,2,\ldots,r$, $i\neq j$.
We get that  $\textbf{e}_{V_1}, \ldots, \textbf{e}_{V_r}$ is linearly independent eigenvectors corresponding to $0$ of $\mathcal{L}_{\mathcal{G}}$. Next, we prove $\textbf{e}_{V_1}, \ldots, \textbf{e}_{V_r}$ is a maximal linearly independent group of nonnegative eigenvectors    corresponding to  eigenvalue $0$ of $\mathcal{L}_{\mathcal{G}}$, i.e., $\beta(\mathcal{G})=r$.

\vspace{3mm}

For any  nonnegative eigenvector $x=(x_1,\ldots,x_n)^{\mathrm{T}}$  corresponding to  eigenvalue $0$ of $\mathcal{L}_{\mathcal{G}}$, we have $\mathcal{L}_{\mathcal{G}}x^{k-1}=0 \text{ and } x\in \mathbb{R}_+^n$.
 Let  ${x}_{V_i}=(x^{(i)}_j)\in \mathbb{R}^n$, its entry $x^{(i)}_j=x_j$ if $j\in V_i$ and zero otherwise, $i=1,\ldots,r$. So $x=x_{V_1}+\cdots +x_{V_r}$.
 Since  $\mathcal{G}(V_1),\ldots,\mathcal{G}(V_r)$ are the connected components of ${\mathcal{G}}$, we get
 \[
  \mathcal{L}_{\mathcal{G}}(x_{V_i})^{k-1}=0,~ i=1,\ldots,r.
  \]
Then,
$$\mathcal{L}_{\mathcal{G}}[V_i](x[V_i])^{k-1}=0,~ i=1,\ldots,r.$$
So,   ${x}[V_i]$ is a nonnegative eigenvector corresponding to  eigenvalue $0$ of $\mathcal{L}_{\mathcal{G}}[V_i]$ if ${x}[V_i]\neq 0 $, $i\in[r]$.
From Lemma \ref{th1}, we know that  the number of maximal linearly independent nonnegative eigenvectors of ${\mathcal{G}}(V_i)$  corresponding to  eigenvalue $0$ of $\mathcal{L}_{\mathcal{G}}[V_i]$ is 1.
Thus,  ${x}[V_i] = c_i\textbf{e}[V_i] $, i.e., ${x}_{V_i} = c_i\textbf{e}_{V_i} $, $c_i$ is a constant, $i=1,\ldots,r$.
Therefore,  $x$ is a  linear combination of $\textbf{e}_{V_1}, \ldots, \textbf{e}_{V_r}$. So  $\textbf{e}_{V_1}, \ldots, \textbf{e}_{V_r}$ is a maximal linearly independent group of  nonnegative eigenvectors
  corresponding to  eigenvalue $0$ of $\mathcal{L}_{\mathcal{G}}$, i.e., $\beta(\mathcal{G})=r$.
\end{proof}

\vspace{5mm}

By Theorem \ref{th2}, we can get the following result directly.

\begin{thm}\label{th3}
A $k$-uniform hypergraph  ${\mathcal{G}}$ is connected if and only if $\beta(\mathcal{G})=1$.
\end{thm}


Obviously, the maximum numbers of linearly independent vectors  of  sets                                                                          $\{x: \mathcal{L}_{\mathcal{G}}x^{k-1}=0 \text{ and } x\in \mathbb{R}_+^n\}$ and   $\{x: \mathcal{L}_{\mathcal{G}}x^{k-1}=0, x^{\mathrm{T}}x=1 \text{ and } x\in \mathbb{R}_+^n\}$ are equal, i.e., $\beta(\mathcal{G})=\beta_Z(\mathcal{G})$.  By Theorem \ref{th2}, we have the following result.

\begin{thm}
Let ${\mathcal{G}}$ be a $k$-uniform hypergraph. Then the number of connected components of ${\mathcal{G}}$ is the Z-geometry connectivity $\beta_Z(\mathcal{G})$.
\end{thm}

%

\vspace{3mm}

Let ${\mathcal{G}}$ be a $k$-uniform hypergraph and  $\mathcal{A}_{\mathcal{G}}$ be its adjacency tensor.
Let $\beta_{\rho}(\mathcal{G})$  denote the  maximum number of linearly independent nonnegative eigenvectors of  $\mathcal{A}_{\mathcal{G}}$ corresponding to spectral radius $\rho(\mathcal{A}_{\mathcal{G}})$.
We have the following conclusion.
\begin{thm}\label{th5}
Let ${\mathcal{G}}$ be a $k$-uniform $d$-regular hypergraph. Then  the number of connected components of ${\mathcal{G}}$ is $\beta_{\rho}(\mathcal{G})$ .
\end{thm}
\begin{proof}
It's easy to check that $(d,\textbf{e})$ is an eigenpair of $\mathcal{A}_{\mathcal{G}}$, where $\textbf{e}=(1,1,\ldots,1)^{\mathrm{T}}\in \mathbb{R}^n$.
From Lemma \ref{non}, $d$ is the spectral radius of ${\mathcal{A}}_{\mathcal{G}}$.

Since ${\mathcal{G}}$ is a  $d$-regular hypergraph, we have its Laplacian tensor is $\mathcal{L}_{\mathcal{G}} =d \mathcal{I} -\mathcal{A}_{\mathcal{G}}$. Thus, $(\lambda,x)$ is an eigenpair of $\mathcal{A}_{\mathcal{G}}$ if and only if $(d-\lambda,x)$ is an eigenpair of $\mathcal{L}_{\mathcal{G}}$.
So, $(d,x)$ is an eigenpair of $\mathcal{A}_{\mathcal{G}}$ if and only if $(0,x)$ is an eigenpair of $\mathcal{L}_{\mathcal{G}}$.
Thus, $\beta_{\rho}(\mathcal{G})=\beta(\mathcal{G})$.

By Theorem \ref{th2}, the statement  holds.
\end{proof}
%

\begin{rem}
Let ${\mathcal{G}}$ is an ordinary graph with r connected components. Since Laplacian matrix $\mathcal{L}_{\mathcal{G}}$ is symmetric, the algebraic multiplicity and the geometry  multiplicity of Laplacian eigenvalue 0 are equal, are both r. It's easy to know that we can choose r nonnegative eigenvectors as the basis of characteristic subspace of Laplacian eigenvalue 0. Therefore, when ${\mathcal{G}}$ is an ordinary graph, Theorem \ref{th2} is the  Theorem \ref{Fiedler2}, and Theorem \ref{th3} is the Fiedler's result Theorem \ref{Fiedler1}, respectively.

Similarly,  when ${\mathcal{G}}$ is a $d$-regular ordinary graph with r connected components, we can choose r nonnegative eigenvectors as the basis of characteristic subspace of spectral radius. Thus, by Theorem \ref{th5}, we get that the result of regular graphs ``the number of connected components of $\mathcal{G}$ is equal to the algebraic multiplicity of  spectral radius   of the adjacency matrix \cite{introduction,Spectra-of-graphs}".

When ${\mathcal{G}}$ is a $k$-uniform  hypergraph, the maximum number of linearly independent nonnegative eigenvectors and the maximum number of linearly independent  eigenvectors corresponding to  $0$ of $\mathcal{L}_{\mathcal{G}}$ isn't equal.
For example, for $4$-uniform hypergragh $\mathcal{G}$ in Figure 1, it's easy to check that $(1,1,1,1)^{\mathrm{T}}$, $(1,1,-1,-1)^{\mathrm{T}}$, $(1,-1,1,-1)^{\mathrm{T}}$ and $(1,-1,-1,1)^{\mathrm{T}}$ are the linearly independent  eigenvectors  of $\mathcal{L}_{\mathcal{G}}$ corresponding to  $0$.
Thus, the geometry connectivity $\beta(\mathcal{G})$ isn't equal to the maximum number of linearly independent  eigenvectors corresponding to  $0$ of $\mathcal{L}_{\mathcal{G}}$.

\end{rem}

\end{spacing}


\vspace{3mm}
\noindent
\textbf{References}
\bibliographystyle{elsarticle-num}

\end{CJK*}
\end{document}